\documentclass[a4paper,10pt]{amsart}
\usepackage{amsmath,amssymb,amsthm}
\usepackage[bookmarks=true,hyperindex,pdftex,colorlinks,citecolor=red, linkcolor=cyan,urlcolor=cyan]{hyperref}
\usepackage{amsfonts,amsmath, amsthm, amscd, amssymb, graphicx, color, mathrsfs, stmaryrd}
\usepackage[utf8]{inputenc}
\usepackage{tikz-cd}
\usepackage{adjustbox}
\usepackage{marginnote} 
\usepackage[all]{xy} 
\usepackage{tikz-cd}
\usetikzlibrary{arrows.meta}
\tikzcdset{
  arrow style={tikz,diagrams={>=Triangle}}
}

\usepackage[a4paper,lmargin=3cm,rmargin=3cm,tmargin=4cm,bmargin=4cm,marginparwidth=2.8cm,marginparsep=1mm]{geometry} 

\usepackage{ulem} 
\usepackage{bbm}

\newcommand{\norm}[1]{\left\Vert#1\right\Vert}
\newcommand{\set}[1]{\left\{{#1}\right\}}     
\newcommand{\abs}[1]{\lvert#1\rvert}
\newcommand{\restricted}{\mathord{\upharpoonright}}
\newcommand{\Free}{{\mathcal F}}
\newcommand{\lipfree}[1]{\mathcal{F}({#1})}  
\newcommand{\Lip}{{\mathrm{Lip}}_0}
\newcommand{\meas}[1]{\mathcal{M}({#1})} 
\newcommand{\R}{\mathbb{R}}

\renewcommand{\geq}{\geqslant}
\renewcommand{\leq}{\leqslant}

\DeclareMathOperator{\dens}{dens}

\newtheorem{thm}{Theorem}[section]

\newtheorem{prop}[thm]{Proposition}
\newtheorem{coro}[thm]{Corollary}
\newtheorem{lema}[thm]{Lemma}

\newtheorem{ques}[thm]{Question}
\theoremstyle{definition}

\newtheorem{ejem}[thm]{Example}

\numberwithin{equation}{section}

\begin{document}

\title{A note on nonseparable Lipschitz-free spaces}

\author[R. J. Aliaga]{Ram\'on J. Aliaga}
\address[R. J. Aliaga]{Instituto Universitario de Matem\'atica Pura y Aplicada,
Universitat Polit\`ecnica de Val\`encia,
Camino de Vera S/N,
46022 Valencia, Spain}
\email{raalva@upv.es}

\author[G. Grelier]{Guillaume Grelier}
\address[G. Grelier]{Universidad de Murcia, Departamento de Matem\'aticas, Campus de Espinardo, 30100 Murcia, Spain} \email{g.grelier@um.es}

\author[A. Proch\'azka]{Anton\'in Proch\'azka}
\address[A. Proch\'azka]{
Universit\'e de Franche Comt\'e, CNRS, LmB (UMR 6623), F-25000 Besan\c con, France
}
\email{antonin.prochazka@univ-fcomte.fr}



\begin{abstract}
We prove that several classical Banach space properties are equivalent to separability for the class of Lipschitz-free spaces, including Corson's property ($\mathcal{C}$), Talponen's Countable Separation Property, or being a G\^ateaux differentiability space. On the other hand, we single out more general properties where this equivalence fails. In particular, the question whether the duals of non-separable Lipschitz-free spaces have a weak$^*$ sequentially compact ball is undecidable in ZFC. Finally, we provide an example of a nonseparable dual Lipschitz-free space that fails the Radon-Nikod\'ym property.
\end{abstract}

\maketitle

\section{Introduction}

The main goal of this note is to show that certain well-known Banach space properties, many of which can be interpreted as generalizations of separability, are equivalent to separability within the class of Lipschitz-free spaces. Our study focuses on properties relating to differentiability of convex functions (e.g. being weak Asplund, or a G\^ateaux differentiability space), topological properties of the dual ball (e.g. being Eberlein, or angelic, or sequentially compact, with respect to the weak$^*$ topology) and structural properties such as being weakly compactly generated (WCG), weakly countably determined (WCD), or weakly Lindel\"of determined (WLD). Some of these equivalences are already known: for instance, weakly compact subsets of Lipschitz-free spaces were already shown to be separable by Kalton in \cite{Kalton}. The takeaway is that nonseparable Lipschitz-free spaces are, in a sense, very far from being separable in comparison to general Banach spaces.

Section \ref{sec:separability} is devoted to proving these equivalences. Many of them are straightforward and boil down to the fact that nonseparable Lipschitz-free spaces contain $\ell_1(\omega_1)$, a fact first observed by H\'ajek and Novotn\'y \cite{HajekNovotny}, but others are more subtle. For instance, the equivalence with Talponen's countable separation property depends on a non-trivial intersection property of families of Lipschitz-free spaces due to Aliaga and Perneck\'a \cite{AP}. Equivalence between separability and weak$^*$ sequential compactness of the dual ball turns out to be undecidable in ZFC. We isolate the precise axiomatic condition under which equivalence holds; in particular, it holds under the Continuum Hypothesis. We also briefly discuss the questions of which Lipschitz-free spaces have a weak$^*$ separable dual, or dual ball (which are never equivalent to separability).

Finally, in Section \ref{sec:rnp} we prove that spaces of measures on Polish spaces can be realized as Lipschitz-free spaces. This yields an example of a nonseparable Lipschitz-free space, namely $C([0,1])^*$, that admits an isometric predual but fails the Radon-Nikod\'ym property, and moreover fails to have a predual that consists entirely of locally flat functions. Whether the latter is possible in the separable setting is still an open problem.

Our notation is standard and follows textbooks such as \cite{FHHMZ}. Throughout the document, the symbol $X$ is used for a real Banach space, and its unit ball is denoted by $B_X$. Likewise, $(M,d)$ is always a complete metric space that we assume to be pointed, that is, it has a distinguished point $0$ (the choice of which is inconsequential), and we denote by $\text{dens}(M)$ the density character of $M$, that is, the minimum cardinality of a dense subset of $M$. Let us briefly recall the construction of Lipschitz-free spaces. We consider the vector space $\Lip(M)$ of Lipschitz functions $f:M\to\R$ such that $f(0)=0$. Endowed with the norm
$$
\norm{f} = \sup\set{ \frac{\abs{f(x)-f(y)}}{d(x,y)}\ :\ x\neq y\in M },
$$
$\Lip(M)$ is a Banach space. If $x\in M$, we define $\delta(x):M\to\Lip(M)^*$ by $\delta(x)(f)=f(x)$. It is easy to see that $\delta:M\to \Lip(M)^*$ is an isometry. The Lipschitz-free space over $M$ is then $\lipfree{M}=\overline{\text{span}}\,\delta(M)$. This is a Banach space with the norm inherited from $\Lip(M)^*$ that satisfies the following linearization property: any Lipschitz mapping $f:M\to X$ such that $f(0)=0$ can be extended to a unique operator $\widehat{f}:\lipfree{M}\to X$ such that $\widehat{f}\circ\delta=f$, and $||\widehat{f}||$ is the Lipschitz constant of $f$. In particular, $\lipfree{M}^*=\Lip(M)$, and bi-Lipschitz equivalent metric spaces have isomorphic Lipschitz-free spaces. Moreover, if $0\in N\subset M$ then $\lipfree{N}$ can be identified with the subspace $\overline{\text{span}}\,\delta(N)$ of $\lipfree{M}$. It is elementary that $\lipfree{M}$ is separable if and only if $M$ is separable, if and only if $M$ does not contain an uncountable subset that is uniformly discrete (i.e. for which non-zero distances have a strictly positive lower bound); we will use this repeatedly without mention. For further reference on Lipschitz-free spaces, we refer the reader to \cite{Weaver2} (where they are referred to as \textit{Arens-Eells spaces}).

Last, we recall a construction that will feature several times in this note. The \textit{metric sum} of a family $\set{(M_\gamma,d_\gamma):\gamma\in\Gamma}$ of pointed metric spaces is defined as the metric space $M=\{0\}\cup\bigcup_{\gamma\in\Gamma}(M_\gamma\setminus\{0\})$ with the metric given by $d(x,x')=d_\gamma(x,x')$ for $x,x'\in M_\gamma$ and $d(x,x')=d_\gamma(x,0)+d_{\gamma'}(0,x')$ for $x\in M_\gamma$, $x'\in M_{\gamma'}$. That is, $M$ is obtained by taking the disjoint union of the $M_\gamma$'s, identifying their base points, and having all paths between different $M_\gamma$'s go through the base point. $\lipfree{M}$ is then linearly isometric to the $\ell_1$-sum of the spaces $\lipfree{M_\gamma}$ (see e.g. \cite[Proposition 3.9]{Weaver2}).

\section{Properties equivalent to separability}
\label{sec:separability}

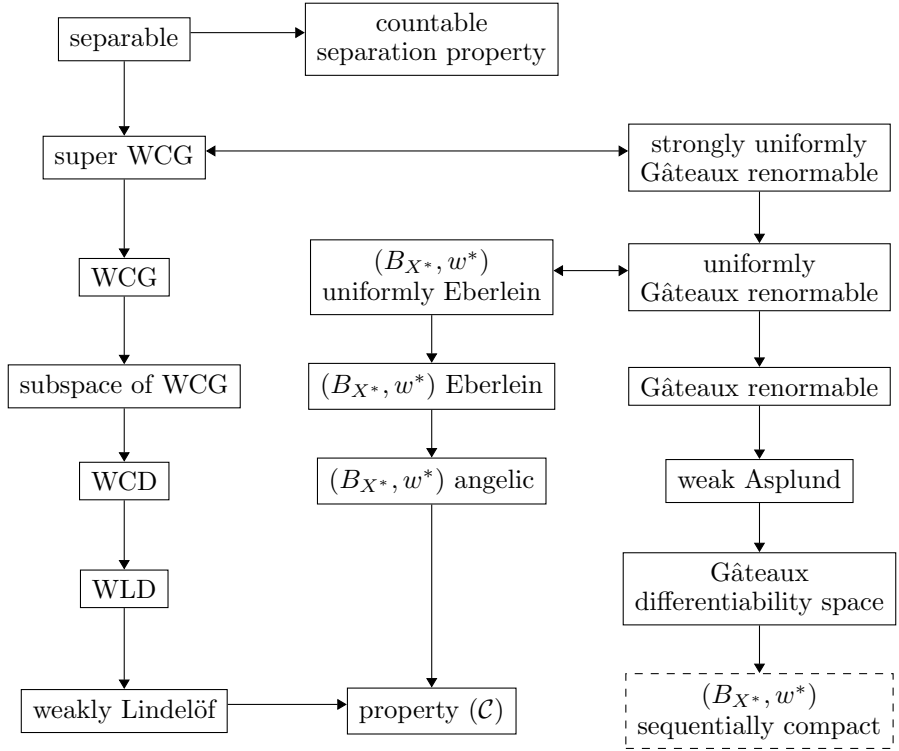
\begin{figure*}
  \centering
  \begin{tikzcd}[cells={nodes={draw}}]
    \txt{separable} \ar[d] \ar[r]
    & \txt{countable\\separation property} \\
    \txt{super WCG} \ar[rr,leftrightarrow] \ar[d]
    & & \txt{strongly uniformly\\G\^ateaux renormable} \ar[d] \\
    \txt{WCG} \ar[d]
    & \txt{$(B_{X^*},w^*)$\\uniformly Eberlein} \ar[d]
    & \txt{uniformly\\G\^ateaux renormable} \ar[d] \ar[l,leftrightarrow] \\
    \txt{subspace of WCG} \ar[d]
    & \txt{$(B_{X^*},w^*)$ Eberlein} \ar[d]
    & \txt{G\^ateaux renormable} \ar[d] \\
    \txt{WCD} \ar[d]
    & \txt{$(B_{X^*},w^*)$ angelic} \ar[dd]
    & \txt{weak Asplund} \ar[d] \\
    \txt{WLD} \ar[d]
    & 
    & \txt{G\^ateaux\\differentiability space} \ar[d] \\
    \txt{weakly Lindel\"of} \ar[r]
    & \txt{property ($\mathcal{C}$)}
    & |[dashed]| \txt{$(B_{X^*},w^*)$\\sequentially compact}
  \end{tikzcd}
  \caption{These implications hold for a general Banach space $X$.}
  \label{fig:implications}
\end{figure*}

The chart in Figure \ref{fig:implications} includes several well-known Banach space properties, all of which hold for separable Banach spaces, together with some of the implications between them. References for the properties and the implications between them will be provided below as they are addressed. We stress that the list of shown implications is by no means exhaustive, but it will be enough for our purposes. The main goal of this section is to prove the following theorem:

\begin{thm}
For a Lipschitz-free space, all properties in solid boxes in Figure \ref{fig:implications} are equivalent to separability. Assuming the Continuum Hypothesis, the property in the dashed box is also equivalent.
\end{thm}

Because of the many implications that hold for general Banach spaces, it will suffice to prove the equivalence to separability of a few properties at the end of the implication chains, namely:
\begin{itemize}
\item being a G\^ateaux differentiability space,
\item Corson's property ($\mathcal{C}$),
\item having a weak$^\ast$ sequentially compact dual ball (assuming CH), and
\item the Countable Separation Property.
\end{itemize}
These will be proved in separate subsections below; see Proposition \ref{GDS-Corson}, Corollary \ref{weakstarseq_lip}, and Proposition \ref{csp}.

\subsection{Differentiability properties and Corson's property (\texorpdfstring{$\mathcal{C}$}{C})}
\label{sub:differentiability}

We consider first properties relating to the differentiability of convex functions in Lipschitz-free spaces in terms of separability. A Banach space $X$ is a \textit{G\^ateaux differentiability space} (GDS) if every continuous convex function defined on an open convex subset $U$ of $X$ is G\^ateaux differentiable on a dense subset of $U$. If that dense set can be moreover chosen to be $G_\delta$, then $X$ is \textit{weak Asplund}.
It has been shown in \cite{GDS} that the class of G\^ateaux differentiability spaces is strictly larger than the class of weak Asplund spaces.

On the other hand, recall that $X$ has \textit{Corson's property} ($\mathcal{C}$) if for every family $\mathcal{A}$ of closed convex sets in $X$ with empty intersection there is a countable subfamily $\mathcal{B}$ of $\mathcal{A}$ with empty intersection.

\begin{prop}\label{GDS-Corson}
If a Lipschitz-free space is a GDS, or has Corson's property ($\mathcal{C}$), then it is separable.
\end{prop}

\begin{proof}
Suppose that $\lipfree{M}$, or equivalently $M$, is not separable. By \cite[Proposition 3]{HajekNovotny}, $\lipfree{M}$ contains a complemented subspace isomorphic to $\ell_1(\omega_1)$.

The norm of $\ell_1(\omega_1)$ is nowhere G\^ateaux differentiable (see \cite[Example 1.6.c]{DGZ}), hence $\ell_1(\omega_1)$ is not a GDS. Since the property of being a GDS is inherited by complemented subspaces (see for example \cite[Proposition 6.8]{Phelps}) and is invariant under isomorphism, we conclude that $\lipfree{M}$ is not a GDS. 

Similarly, property ($\mathcal{C}$) is stable under isomorphism and passing to subspaces, and it is known that $\ell_1(\omega_1)$ does not have this property (see \cite[Exercise 14.46]{FHHMZ}), so $\lipfree{M}$ must fail property ($\mathcal{C}$). Alternatively, since $\dens(C[0,\omega_1])=\aleph_1$, there is a quotient map from $\ell_1(\omega_1)$ onto $C[0,\omega_1]$, and the latter is known to fail Corson's property ($\mathcal C$) (see \cite[Theorem 14.36]{FHHMZ}) and this property is stable by quotients.
\end{proof}

Proposition \ref{GDS-Corson} already yields the equivalence of most of the properties in Figure \ref{fig:implications} for Lipschitz-free spaces. For instance, $\Free(M)$ is separable if and only if it admits an equivalent ((strongly) uniformly) G\^ateaux differentiable norm. Indeed, existence of such norm on a space $X$ implies that $X$ is weak Asplund \cite[Corollary 4.2.5]{Fabian}, and hence a GDS. Note that a Lipschitz-free space $\lipfree{M}$ can only be Asplund if it is finite-dimensional, that is, if $M$ is finite, as any infinite-dimensional Lipschitz-free space contains an isomorphic copy of $\ell_1$ \cite[Proposition 3]{HajekNovotny}, which is not Asplund. In particular, $\Lip(M)$ is WCG if and only if $M$ is finite (see \cite[Theorem 2.43]{Phelps}).

On the other hand, the equivalence with property ($\mathcal{C}$) yields equivalence for the properties in the left column of Figure \ref{fig:implications} including, in order of increasing generality, being WCG, a subspace of a WCG space, WCD, WLD, and weakly Lindel\"of (see e.g. Theorem 3.8 in \cite{Handbook} and the paragraph preceding it). Let us also mention super WCG spaces, a notion that was introduced in \cite{Raja} and shown to agree with strong uniform G\^ateaux renormability. Similarly, by combining Proposition \ref{GDS-Corson} and well known implications, separability of $\Free(M)$ is equivalent to any of the following properties of $(B_{\Lip(M)},w^*)$, in increasing order of generality: uniform Eberlein, Eberlein and angelic. Indeed, $X$ admits an equivalent uniformly G\^ateaux norm if and only if $(B_{X^*},w^*)$ is uniformly Eberlein \cite[Theorem 14.16]{FHHMZ}, and if $(B_{X^*},w^*)$ is angelic then $X$ has property ($\mathcal{C}$) by \cite[Theorem 14.37]{FHHMZ}.

\subsection{Weak\texorpdfstring{$^*$}{*} sequential compactness of the dual ball}

Since every GDS has weak$^*$ sequentially compact dual ball (see \cite[Theorem 2.1.2]{Fabian}), it is natural to ask whether $\Free(M)$ has to be separable if $B_{\Lip(M)}$ is weak$^*$ sequentially compact.
It turns out that this question is undecidable in ZFC.

We denote by $[\mathbb N]^\omega$ the family of all infinite subsets of $\mathbb N$. Using the notation from \cite{vanDouwenHB}, we say that a subset $\mathscr{S}\subset [\mathbb N]^\omega$ is a \textit{splitting family} if it satisfies the following: for any $A\in [\mathbb N]^\omega$ there is $S\in\mathscr{S}$ such that $A\cap S$ and $A\setminus S$ are both infinite. The \textit{splitting cardinal} $\mathfrak{s}$ is defined as the smallest possible cardinality of a splitting family in $[\mathbb N]^\omega$. It is easy to check that $\aleph_1\leq\mathfrak{s}\leq\mathfrak{c}$, and the four possibilities $\aleph_1=\mathfrak{s}=\mathfrak{c}$ (i.e. the Continuum Hypothesis), $\aleph_1=\mathfrak{s}<\mathfrak{c}$, $\aleph_1<\mathfrak{s}=\mathfrak{c}$, $\aleph_1<\mathfrak{s}<\mathfrak{c}$ are all known to be consistent in ZFC \cite[Theorem 5.1]{vanDouwenHB}.

\begin{thm}\label{weakstarseq}
Let $M$ be a complete metric space. Then the following are equivalent:
\begin{enumerate}
    \item[(i)] $\dens(M)\geq\mathfrak{s}$,
    \item[(ii)] $M$ contains a uniformly discrete subset of cardinality $\mathfrak{s}$,
    \item[(iii)] $B_{\Lip(M)}$ is not $w^*$-sequentially compact.
\end{enumerate}
\end{thm}

Before we prove this theorem, let us state two immediate consequences. First, it is clear that condition (i) is equivalent to nonseparability precisely when $\mathfrak{s}=\aleph_1$, therefore:

\begin{coro}\label{weakstarseq_lip}
There exist nonseparable Lipschitz-free spaces $\lipfree{M}$ such that $B_{\Lip(M)}$ is $w^*$-sequentially compact if and only if $\mathfrak{s}\neq\aleph_1$.

In particular, under the Continuum Hypothesis, $B_{\Lip(M)}$ is $w^*$-sequentially compact if and only if $\lipfree{M}$ is separable.
\end{coro}

Second, since $\lipfree{M}$ is isomorphic to $\ell_1(\abs{M})$ when $M$ is infinite and endowed with a metric in which any pair of different points are at distance $1$, we deduce as a particular case:

\begin{coro}
$B_{\ell_\infty(\Gamma)}$ is $w^*$-sequentially compact if and only if $\abs{\Gamma}<\mathfrak{s}$.
Therefore it is undecidable in ZFC whether $B_{\ell_\infty(\omega_1)}$ is w$^*$-sequentially compact.
\end{coro}

For the proof of Theorem \ref{weakstarseq} we will use the following easy observation.

\begin{lema}\label{l:bigSeparatedSet}
Let $\kappa$ be a cardinal of uncountable cofinality. Then $\dens(M)\geq \kappa$ if and only if $M$ contains a uniformly discrete subset of cardinality $\kappa$.
\end{lema}
\begin{proof}
Assume $\dens(M)\geq \kappa$. For each $n\in\mathbb N$, let $A_n$ be a maximal $\frac{1}{n}$-separated subset of $M$. Then $\bigcup_{n\in\mathbb N}A_n$ is dense in $M$, hence it has cardinality at least $\kappa$. 
Since $\kappa$ has uncountable cofinality, this implies that $\abs{A_n}\geq\kappa$ for some $n$.

For the converse, let $A\subset M$ be an $\varepsilon$-separated set of cardinality $\kappa$ and let $D$ be a dense set in $M$. For every $x\in A$ fix some $\psi(x) \in D \cap B(x,\varepsilon/3)$. Then $\psi$ is injective and $\abs{D}\geq\abs{A}=\kappa$.
\end{proof}

Note that the direct implication of Lemma \ref{l:bigSeparatedSet} fails if $\kappa > \aleph_0$ has countable cofinality: suppose that $\kappa=\sup_n\kappa_n$ with $\kappa>\kappa_n$ for all $n$, let $M_n$ be a metric space of cardinality $\kappa_n$ where all non-zero distances are $\frac{1}{n}$, and construct $M$ as a separated union of the spaces $M_n$. Then $\dens(M)=\kappa$ but $M$ contains no uniformly discrete subset of cardinality $\kappa$.

\begin{proof}[Proof of Theorem \ref{weakstarseq}]
(i)$\Leftrightarrow$(ii): 
It is easy to prove that the cofinality of $\mathfrak{s}$ is not countable (see e.g. \cite[Proposition 1.1]{DowShelah}), so this follows from Lemma \ref{l:bigSeparatedSet}.

(ii)$\Rightarrow$(iii): Fix a set $\set{0}\cup\set{x_\gamma : \gamma<\mathfrak{s}}\subset M$ such that any pair of points are at distance at least $r>0$, and define functions $f_\gamma\in B_{\Lip(M)}$ by
$$
f_\gamma(x) = \max\set{r-d(x,x_\gamma),0}
$$
for $x\in M$, so that $f_\gamma(x_\gamma)=r$ and $f_\gamma(x_\lambda)=0$ for $\lambda<\mathfrak{s}$, $\lambda\neq\gamma$. Now fix a splitting family $\mathscr{S}=\set{S_\gamma:\gamma<\mathfrak{s}}$ in $[\mathbb N]^\omega$ and define a sequence $(g_n)_n$ in $B_{\Lip(M)}$ by $g_n = \max\set{f_\gamma: n\in S_\gamma}$ so that
$$
g_n(x_\gamma) = \begin{cases}
r &\text{if $n\in S_\gamma$} \\
0 &\text{if $n\notin S_\gamma$}
\end{cases} .
$$
Let us see that no subsequence $(g_{n_k})_{k\in\mathbb N}$ of $(g_n)$ converges weak$^*$, i.e. pointwise. Let $A=\set{n_k:k\in\mathbb N}\in[\mathbb N]^\omega$, then there exists $\gamma<\mathfrak{s}$ such that $A\cap S_\gamma$ and $A\setminus S_\gamma$ are both infinite, thus $g_{n_k}(x_\gamma)$ takes both values $0$ and $r$ for infinitely many values of $k$ so it cannot converge. We conclude that $B_{\Lip(M)}$ is not $w^*$-sequentially compact.

(iii)$\Rightarrow$(i): We may assume that $M$ is infinite, as the implication is trivial otherwise. Suppose that there is a sequence $(f_n)_n$ in $B_{\Lip(M)}$ with no pointwise convergent subsequence. Let $D$ be a dense subset of $M$ and $\mathscr{S}$ be the family of sets $S_{x,q}\subset\mathbb N$, for $x\in D$ and $q\in\mathbb Q$, given by
$$
S_{x,q} = \set{n\in\mathbb N : f_n(x)>q} .
$$
Fix an arbitrary $A\in[\mathbb N]^\omega$ and order it as $A=\{n_k:k\in\mathbb N\}$ with $n_k<n_{k+1}$ for all $k$.
Since $(f_{n_k})_k$ is equi-Lipschitz, there is $x \in D$ such that the sequence $(f_{n_k}(x))_k$ does not converge (here we use the easily provable and well known fact that if a sequence of equi-Lipschitz functions converges pointwise on a dense set, then it converges pointwise everywhere).
Because $f_{n_k}(x)$ is bounded by $d(x,0)$, there must be two different cluster points $a<b$. If $q \in \mathbb Q$ is such that $a<q<b$ then the set $S_{x,q}$ splits the set $A$. This shows that $\mathscr{S}$ contains a splitting family in $[\mathbb N]^\omega$ and therefore $\abs{D}= \abs{D\times\mathbb Q}\geq\abs{\mathscr{S}}\geq\mathfrak{s}$.
\end{proof}

\subsection{Weak\texorpdfstring{$^*$}{*} separability of the dual}

In \cite{Talponen}, Talponen introduced the following property: a Banach space $X$ has the \textit{Countable Separation Property} (CSP) if any set $A\subset X^*$ that separates points of $X$ has a countable subset $B\subset A$ that also separates points of $X$. It is obvious that separable Banach spaces have the CSP. 

\begin{prop}\label{csp}
A Lipschitz-free space has the CSP if and only if it is separable.
\end{prop}

\begin{proof}
Let $M$ be a complete pointed metric space. Suppose that $M$ is not separable and let $(B_\gamma)_{\gamma<\omega_1}$ be a family of closed balls in $M$ with $\inf_{\gamma\neq\gamma'} d(B_\gamma,B_{\gamma'})>0$. Define
$$
A_\gamma:=\{0\}\cup\bigcup\set{B_\lambda:\gamma\leq\lambda<\omega_1} .
$$
Then $(A_\gamma)_{\gamma<\omega_1}$ is a strictly decreasing family of closed subsets of $M$ whose intersection is $\{0\}$. Therefore $(\lipfree{A_\gamma})_{\gamma<\omega_1}$ is a strictly decreasing family of closed subspaces of $\lipfree{M}$ whose intersection is
$$
\bigcap_{\gamma<\omega_1} \lipfree{A_\gamma} = \mathcal{F}\left(\bigcap_{\gamma<\omega_1} A_\gamma\right) = \lipfree{\{0\}} = \{0\}
$$
where the first equality follows from \cite[Theorem 2.1]{APPP}. By \cite[Theorem 4.1]{Talponen}, this implies that $\lipfree{M}$ fails the CSP.
\end{proof}

Any Banach space with the CSP clearly has weak$^*$ separable dual ball. In contrast to Proposition \ref{csp}, the $w^*$-separability of $B_{\Lip(M)}$ does not imply the separability of $\lipfree{M}$. In fact, $B_{\Lip(\ell_\infty)}$ is $w^*$-separable whereas $\Free(\ell_\infty)$ is nonseparable. More precisely, we have:

\begin{prop}\label{p:weak-star-separable-dual-ball}
Let $M$ be a pointed metric space. Then the following are equivalent:
\begin{enumerate}
    \item[(i)] $M$ isometrically embeds into $\ell_\infty$;
    \item[(ii)] $B_{\Lip(M)}$ is $w^*$-separable;
    \item[(iii)] $\lipfree{M}$ is isometric to a subspace of $\ell_\infty$.
\end{enumerate}
\end{prop}

\begin{proof}
The equivalence between (ii) and (iii) is true for any Banach space and can be found in \cite{Dancer}.
Let us show (ii)$\Rightarrow$(iii) explicitly, for the convenience of the reader. Let $(f_n)_{n=1}^\infty\subset B_{\Lip(M)}$ be $w^*$-dense in $B_{\Lip(M)}$. Then $u:\Free(M) \to \ell_\infty$ defined by $u(\mu)=(\langle\mu,f_n\rangle)_{n=1}^\infty$ when $\mu\in\Free(M)$ is easily checked to be an isometric embedding.
To see that (i)$\Rightarrow$(iii), note that $\Free(M)\subseteq \Free(\ell_\infty)$ and $\Free(\ell_\infty)$ is isometric to a subspace of $\ell_\infty$ as proved by Kalton in \cite[Proposition 5.1]{Kalton}.
Finally, since $M$ is isometric to a subset of $\lipfree{M}$, it is clear that (iii)$\Rightarrow$(i), and the proof is complete.
\end{proof}

The next proposition has essentially the same proof as the previous one.

\begin{prop}\label{p:weak-star-separable-dual-ball-for-equivalent-norm}
Let $M$ be a pointed metric space. Then the following are equivalent:
\begin{enumerate}
    \item[(i)] $M$ bi-Lipschitz embeds into $\ell_\infty$;
    \item[(ii)] there is an equivalent norm $\abs{\cdot}$ on $\Free(M)$ such that $B_{(\Lip(M),\abs{\cdot})}$ is $w^*$-separable;
    \item[(iii)] $\lipfree{M}$ is isomorphic to a subspace of $\ell_\infty$.
\end{enumerate}
\end{prop}

Note that if $B_{X^*}$ is $w^*$-separable, then $X^*$ also is. However, the converse can fail in a particularly strong way: there is a Banach space $X$ such that $X^*$ is $w^*$-separable and yet there is no equivalent norm on $X$ whose dual unit ball is $w^*$-separable \cite{JoLi}. 
In what follows we will build such an example which is moreover a Lipschitz-free space.

Following \cite{GLPP}, a metric space $M$ is called \textit{Lip-lin injective} (LLI) if for every pointed metric space $N$ and every injective $f \in \Lip(M,N)$ the canonical linearization $\widehat{f}:\Free(M)\to \Free(N)$ is injective. It is proved in~\cite[Corollary 4.14]{GLPP} that uniformly discrete metric spaces are Lip-lin injective.

\begin{prop}\label{p:weak-star-separable-dual}
Let $M$ be a metric space. Then the following are equivalent:
\begin{itemize}
    \item[(i)] $\Lip(M)$ is $w^*$-separable;
    \item[(ii)] there exists a sequence in $\Lip(M)$ that separates points of $\lipfree{M}$;
    \item[(iii)] there exists a bounded one-to-one linear operator $T:\lipfree{M}\to\ell_\infty$.
\end{itemize}
These equivalent conditions also imply the following one
\begin{itemize}
    \item[(iv)] there is a Lipschitz injection of $M$ into $\ell_\infty$; 
\end{itemize}
Moreover, if $M$ is Lip-lin injective then (iv) implies any of (i)--(iii). 
\end{prop}

We stress that \textit{Lipschitz injection} is a map which is Lipschitz and injective and should not be confused with bi-Lipschitz embedding (frequently called Lipschitz embedding in the recent literature).

\begin{proof}
The equivalence between (i), (ii) and (iii) is true for any Banach space in place of $\lipfree{M}$ and can be found in \cite{Dancer}.
The condition (iii) implies (iv) by restricting $T$ to $\delta(M)$.
In order to prove (iv)$\Rightarrow$(iii) let $f:M \to \ell_\infty$ be Lipschitz and injective. 
Then $\widehat{f}:\Free(M)\to \Free(\ell_\infty)$ is injective as $M$ is Lip-lin injective. 
We can thus define $T=S\circ \widehat{f}$ where $S:\Free(\ell_\infty) \to \ell_\infty$ is the linear embedding constructed by Kalton in \cite[Proposition 5.1]{Kalton}.
\end{proof}

For example, if $\abs{M}>\mathfrak{c}$, then (iv) fails and $\Lip(M)$ is not $w^*$-separable.
On the other hand, $\Lip(M)$ is $w^*$-separable for every uniformly discrete metric space of cardinality at most $\mathfrak c$. Indeed, $M$ Lipschitz injects into a bounded uniformly discrete space (just equip $M$ with $\rho(x,y)=\min\set{d(x,y),1)}$) which, in turn, is bi-Lipschitz equivalent to a subset of $2^{\mathbb N}\subset \ell_\infty$. Since $M$ is LLI, the conclusion follows from (iv)$\Rightarrow$(i) above. 

We do not know if the hypothesis of $M$ being LLI can be removed in general from the above proposition. In particular, this proposition does not allow us to decide whether $\Lip(C[0,\omega_1])$ is $w^*$-separable (it is known that $C[0,\omega_1]$ Lipschitz-injects into $\ell_\infty$, see e.g. \cite{Kalton}). We have learned that Leandro Candido and Marek C\'uth have recently proved that $\Lip(C[0,\omega_1])$ is $w^*$-separable (as well as many other Lipschitz spaces).

Finally, concerning the implication (iv)$\Rightarrow$(i)--(iii) one should not forget the following simple observation.

\begin{prop}\label{p:injectivity-on-span}
Suppose that there is a sequence in $\Lip(M)$ that separates points of $M$. Then there is a sequence in $\Lip(M)$ that separates elements of $\mathrm{span}\,\delta(M)$.
\end{prop}

\begin{proof}
The hypothesis implies that there is $f:M \to \ell_\infty$ which is Lipschitz and injective. 
Indeed, let $(f_n)_{n=1}^\infty \subset \Lip(M)$ which separates points of $M$. Then $f:M \to \ell_\infty$ defined by
$$
f(x)=\left(\frac{f_n(x)}{\norm{f_n}}\right)_{n=1}^\infty
$$
is clearly 1-Lipschitz and injective.
Now, let $i:\Free(\ell_\infty) \to \ell_\infty$ be Kalton's isometric embedding from \cite[Proposition 5.1]{Kalton}. 
We do not know if $i\circ \widehat{f}$ is injective but we know that $\widehat{f}\restricted_{\mathrm{span}\,\delta(M)}$ is injective since $\delta(\ell_\infty)$ is linearly independent. So $i\circ \widehat{f}\restricted_{\mathrm{span}\,\delta(M)}$ is injective which is a reformulation of the conclusion.
\end{proof}

We finish by giving an example of a LLI metric space which Lipschitz injects but does not bi-Lipschitz embed into $\ell_\infty$.
Combining Proposition~\ref{p:weak-star-separable-dual} with  Proposition~\ref{p:weak-star-separable-dual-ball-for-equivalent-norm} it is clear that the Lipschitz-free space of any such metric space furnishes the promised example of a Banach space whose dual is $w^*$-separable but does not admit any equivalent norm such that its dual ball is $w^*$-separable.

\begin{ejem}\label{e:WSLLI}
Let $M$ be the metric sum of $([\omega_1]^k,d_k)$, $k \in \mathbb N$, where $[\omega_1]^k$ is the set of sequences of countable ordinals of length $k$ and $d_k$ is Kalton's interlacing graph distance on this set (see \cite[Section 3]{Kalton}).
Then $M$ Lipschitz injects into $\ell_\infty$ but does not admit a bi-Lipschitz embedding into $\ell_\infty$.
Indeed, by \cite[Theorem 3.6]{Kalton}, $M$ does not bi-Lipschitz embed into $\ell_\infty$.
On the other hand $M$ is uniformly discrete and $\abs{M}\leq \mathfrak c$ thus $\Free(M)$ injects into $\ell_\infty$ by the comment after Proposition~\ref{p:weak-star-separable-dual}.
\end{ejem}

\section{Dual Lipschitz-free spaces without the Radon-Nikod\'ym property}
\label{sec:rnp}

The characterization of those metric spaces $M$ for which $\lipfree{M}$ is (isometrically) a dual Banach space remains an important open problem in Lipschitz-free space theory. This issue has been treated e.g. in \cite{Weaver96,Kalton04,GPPR}, and has only recently been solved for compact $M$ in \cite{AGPP}. Several natural subquestions may be asked about the nature of Lipschitz-free spaces $\lipfree{M}$ admitting a predual, including:

\begin{enumerate}
\item[(a)] Must $\lipfree{M}$ have the Radon-Nikod\'ym property (RNP)? Equivalently by \cite[Theorem C]{AGPP}, must $M$ be purely 1-unrectifiable?
\item[(b)] Does there always exist a predual of $\lipfree{M}$ that is made up exclusively of locally flat functions?
\end{enumerate}

Question (a) is motivated by the fact that separable dual spaces have the RNP. In fact, the RNP is equivalent to duality for $\lipfree{M}$ if $M$ is compact \cite[Theorem B]{AGPP}. There are known examples of nonseparable dual Banach spaces without the RNP, but not within the class of Lipschitz-free spaces so far.

Regarding question (b), we recall that a function $f\in\Lip(M)$ is \textit{locally flat} if we have $\lim_{r\to 0}\norm{f|_{B(x,r)}}=0$ for all $x\in M$. If $\lipfree{M}$ is a dual space then the space of all locally flat functions in $\Lip(M)$ is always an isometric predual thereof when $M$ is compact \cite[Theorem B]{AGPP}. In \cite{Kalton04} and later \cite{GPPR}, the non-compact case was studied and sufficient conditions were found under which there is a predual consisting of all locally flat functions that are moreover continuous with respect to a different topology on $M$ (which plays the role of the relative weak$^\ast$ topology). It is a natural suspicion that this may be the general behavior. Note that preduals of Lipschitz-free spaces are in general not unique, so this does not preclude the existence of preduals containing non-locally-flat functions.

In this section we will provide an example of a nonseparable dual Lipschitz-free space that gives a negative answer to both questions above. Our example is a well-known mathematical object: the space $\meas{K}$ of Radon measures on a metrizable compact space $K$. It is not difficult to show that, more generally, $\meas{S}$ is isometrically a Lipschitz-free space for every Polish (i.e. separable and completely metrizable) space $S$.

\begin{prop}\label{freemeasurespace}
Let $S$ be a Polish space. Then $\meas{S}$ is linearly isometric to $\lipfree{M}$ for some metric space $M$. Specifically:
\begin{enumerate}
\item[(a)] if $S$ is finite or countable then $\meas{S}=\ell_1(\abs{S})=\lipfree{M}$ where $M$ is the metric sum of $\abs{S}$ two-point spaces;
\item[(b)] if $S$ is uncountable then
$$
\meas{S} = \ell_1(\mathfrak{c})\oplus_1\left(\bigoplus_{\mathfrak{c}} L_1\right)_1 = \lipfree{M}
$$
where $M$ is the metric sum of $\mathfrak{c}$ two-point spaces and $\mathfrak{c}$ copies of $[0,1]$.
\end{enumerate}
\end{prop}

\begin{proof}[Proof of Proposition \ref{freemeasurespace}]
Part (a) and the second isometry in part (b) are clear, as the Lipschitz-free spaces over $\set{0,1}$ and $[0,1]$ are isometric to $\R$ and $L_1$, respectively. For the first isometry of part (b), recall that $\abs{S}=\mathfrak{c}$ by the Cantor-Bendixson theorem. We follow the argument in \cite[Proposition 4.3.8(iii)]{AlbiacKalton}, with slight changes and additional details in order to obtain the exact form of $\meas{S}$.

Using Zorn's lemma, find a maximal family $(\mu_i)_{i\in I}$ of Borel probability measures on $S$ that contains the set of Dirac measures $\{\delta_x:x\in S\}$, and such that $\mu_i\perp\mu_j$ for $i\neq j\in I$.
Now define a mapping
$$
T:\meas{S}\to\left(\bigoplus_{i\in I}L_1(\mu_i)\right)_1
$$
by letting $(T\lambda)_i$ be the unique $f_i\in L_1(\mu_i)$ such that $d\lambda=f_i\,d\mu_i+d\nu_i$ where $\nu_i\perp\mu_i$; its existence follows from the Radon-Nikod\'ym and Lebesgue decomposition theorems \cite[Theorem 6.10]{Rudin}.
Let $F\subset I$ be finite. Since $(\mu_i)_{i\in F}$ are pairwise mutually singular, there are pairwise disjoint Borel sets $E_i\subset S$ such that $\mu_i$ is concentrated on $E_i$ for all $i\in F$. Then
$$
\sum_{i\in F}\norm{f_i}_{L_1(\mu_i)} = \sum_{i\in F}\int_{E_i}\abs{f_i}\,d\mu_i = \sum_{i\in F}\abs{\lambda}(E_i) \leq \norm{\lambda} .
$$
It follows that $\norm{T\lambda}=\sum_{i\in I}\norm{f_i}_{L_1(\mu_i)}\leq\norm{\lambda}$. Thus $T$ is a linear operator with $\norm{T}\leq 1$.

Reciprocally, define 
$$
U:\left(\bigoplus_{i\in I}L_1(\mu_i)\right)_1 \to \meas{S}
$$
by $U((f_i)_{i\in I})=\lambda'$ such that $d\lambda'=\sum_{i\in I}f_i\,d\mu_i$. Clearly $\norm{\lambda'}\leq\sum_{i\in I}\norm{f_i}_{L_1(\mu_i)}$ as $\norm{\mu_i}=1$ for all $i$, so $\norm{U}\leq 1$. Since the $\mu_i$ are mutually singular, we have $(T\lambda')_i=f_i$ for all $i\in I$, thus $T\circ U$ is the identity.
Finally, let $\lambda \in \meas{S}$ and denote $(T\lambda)_i=f_i$ and $\lambda'=U(T\lambda)$. 
Notice that $(T(\lambda-\lambda'))_i=f_i-f_i=0$ for any $i$, i.e. $\lambda-\lambda'\perp\mu_i$. 
By the maximality of $(\mu_i)$ we get $\lambda-\lambda'=0$, therefore $U\circ T$ is also the identity.
This proves that $T$ and $U$ are inverse onto isometries.

By Theorem 4.13 in \cite{Brezis} and the Example immediately preceding it, every space $L_1(\mu_i)$ is separable. Note that $L_1(\delta_x)=\R$ for each $x\in S$, and all other $\mu_i$ are purely nonatomic by construction, therefore $L_1(\mu_i)=L_1$ by Theorem 14.9 in \cite{Lacey} and its Corollary. The desired isometric identification now follows from the fact that the cardinality of $\{\mu_i:i\in I\}\setminus\{\delta_x:x\in S\}$ is precisely $\mathfrak{c}$ (see e.g. \cite[Exercise 7.14.94]{Bogachev}).
\end{proof}

\begin{coro}\label{c:DualWithoutRNP}
\label{counterexample}
There is a nonseparable Lipschitz-free space $\lipfree{M}$ with the following properties:
\begin{enumerate}
\item[(i)] $\lipfree{M}$ is a dual space;
\item[(ii)] $\lipfree{M}$ admits both separable and nonseparable isometric preduals;
\item[(iii)] $\lipfree{M}$ does not have the Radon-Nikod\'ym property;
\item[(iv)] there is no predual of $\lipfree{M}$ that consists of locally flat functions.
\end{enumerate}
\end{coro}

\begin{proof}
Let $M$ be such that $\lipfree{M}=\meas{[0,1]}$ as given by Proposition \ref{freemeasurespace}. That space is clearly the dual of $C([0,1])$, but $M$ contains an isometric copy $I$ of $[0,1]$ so $\lipfree{M}$ cannot have the RNP because it contains $\lipfree{I}=L_1$ (alternatively, the failure of the RNP for $\lipfree{M}$ follows also from the fact that it is a nonseparable Banach space with a separable predual). Notice also that
$$
\big(C([0,1])\oplus_\infty c_0(\mathfrak{c})\big)^* = \meas{[0,1]} \oplus_1 \ell_1(\mathfrak{c}) = \meas{[0,1]}
$$
so $\lipfree{M}$ also has a nonseparable isometric predual.
For the last statement, note that any locally flat function is constant on $I$ by the fundamental theorem of calculus, therefore locally flat functions cannot separate the functionals $\delta(x)$ for different points $x\in I$.
\end{proof}

Corollary \ref{counterexample} provides negative answers to both questions posed at the beginning of this Section. However, the counterexample is strongly dependent on non-separability. The answer to question (a) is positive in the separable case, so it is natural to ask whether question (b) also has a positive answer for separable $M$:

\begin{ques}
Suppose that $M$ is separable and $\lipfree{M}$ is a dual space. Does there exist a subspace $X\subset\Lip(M)$, containing only locally flat functions, such that $X^*=\lipfree{M}$?
\end{ques}

\section*{Acknowledgements}

We are thankful to Leandro Candido and Marek C\'uth for sharing and discussing their recent unpublished results with us.

Much of this work was carried out during visits of the first and second named authors to the Laboratoire de Math\'ematiques de Besan\c con in 2021.

R. J. Aliaga was partially supported by Grant BEST/2021/080 funded by the Generalitat Valenciana, Spain, and by Grant PID2021-122126NB-C33 funded by MCIN/AEI/10.13039/501100011033 and by ``ERDF A way of making Europe''.

G. Grelier was supported by the grant PID2021-122126NB-C32 funded by MCIN/AEI/10.13039/ 501100011033 and by ``ERDF A way of making Europe'', by the European Union and by the grant 21955/PI/22 funded by Fundaci\'on S\'eneca Regi\'on de Murcia; and by MICINN 2018 FPI fellowship with reference PRE2018-083703, associated to grant MTM2017-83262-C2-2-P.

This work was partially supported by the French ANR project No. ANR-20-CE40-0006.

\end{document}